\documentclass[12pt]{amsart}
\usepackage{amsmath,amssymb,amsthm}    
\usepackage{mathabx}    
\usepackage{mathrsfs, marginnote}    	
\usepackage[backref,colorlinks]{hyperref}
\usepackage[abbrev,msc-links]{amsrefs}  
\usepackage{enumerate}
\usepackage{enumitem}
\usepackage[symbol]{footmisc}
\usepackage[symbol]{footmisc}
\usepackage{pgf,tikz}
\usepackage{float}
\usepackage{mathrsfs}
\usetikzlibrary{arrows}
\usepackage{rotating}
\usetikzlibrary{calc}
\usetikzlibrary{positioning}
\usetikzlibrary{patterns}
\usetikzlibrary{shapes}
\usetikzlibrary{arrows}
\usetikzlibrary{snakes}
\usepackage{caption}
\usepackage{subcaption}

\normalsize                              
\newtheorem{theorem}{Theorem}[section]
\newtheorem{definition}[theorem]{Definition}

\newtheorem{claim}[theorem]{Claim}

\newtheorem{problem}[theorem]{Problem}

\def\eps{\varepsilon}
\def\dHH#1{\leavevmode\setbox0=\hbox{#1}\dimen0=\wd0\setbox0=\hbox{.}%
	\advance\dimen0 by -\wd0%
	\hbox{#1\raise-0.5ex\hbox to 0pt{\hss.\kern.5\dimen0}}}%

\renewcommand{\eps}{{\varepsilon}}

\newcommand{\M}[1]{{G_\varepsilon}(#1)}
\newcommand{\N}{{\mathbb{N}}}
\newcommand{\Sh}[2]{{\textup{Sh}}_{{#1}}^{{#2}}}
\newcommand{\Gc}{G_c}
\newcommand{\Bnd}{B_{\eps}(n,d)}
\newcommand{\fdl}[2]{f_{#1}^{#2}}
\renewcommand{\L}{L}
\renewcommand{\S}{S}

\begin{document}
	\title[Independent sets in subgraphs of a Shift graph]{Independent sets in subgraphs of a Shift graph}
	
	\author[A. Arman]{Andrii Arman}
	\address{Department of Mathematics, University of Manitoba, Winnipeg, MB R3L0J7, Canada}
	\email{andrii.arman@umanitoba.ca}
	\thanks{}
	
	\author[V. R\"{o}dl]{Vojt\v{e}ch R\"{o}dl}
	\address{Department of Mathematics, 
		Emory University, Atlanta, GA 30322, USA}
	\email{rodl@mathcs.emory.edu}
	\thanks{The second  author was supported by NSF grant DMS 1764385}
	
	\author[M. T. Sales]{Marcelo Tadeu Sales}
	\address{Department of Mathematics, 
		Emory University, Atlanta, GA 30322, USA}
	\email{marcelo.tadeu.sales@emory.edu}
	
	%\subjclass[2010]{05C69, 05C63}
	
	%%%%%%%%%%%%%%%%%%%%%%%%%%%%%%%%%%%%%%%%%%%%%%%%%%%%%%%%%%%%
	%%%%%%%%%%%%%%%%%%%%%%%  ABSTRACT   %%%%%%%%%%%%%%%%%%%%%%%%
	%%%%%%%%%%%%%%%%%%%%%%%%%%%%%%%%%%%%%%%%%%%%%%%%%%%%%%%%%%%%
	\maketitle
	\begin{abstract}
		Erd\H{o}s, Hajnal and Szemer\'{e}di proved that any subset $G$ of vertices of a shift graph $\text{Sh}_{n}^{k}$ has the property that the independence number of the subgraph induced by $G$ satisfies $\alpha(\text{Sh}_{n}^{k}[G])\geq \left(\frac{1}{2}-\varepsilon\right)|G|$, where $\varepsilon\to 0$ as $k\to \infty$. In this note we prove that for $k=2$ and $n \to \infty$ there are graphs $G\subseteq \binom{[n]}{2}$ with $\alpha(\text{Sh}_{n}^{2}[G])\leq \left(\frac{1}{4}+o(1)\right)|G|$, and $\frac{1}{4}$ is best possible. We also consider a related problem for infinite shift graphs.
	\end{abstract}
	
	\section{Introduction}
	For $n>k\in \N$ the shift graph $\Sh{n}{k}$ with $$V(\Sh{n}{k})=\{(x_1,\ldots, x_k) \; : \; 1\leq x_1< \ldots < x_k\leq n\}$$ is a graph in which two vertices ${\bf{x}}=(x_1,\ldots, x_k)$ and ${\bf{y}}=(y_1,\ldots, y_k)$ are adjacent if $x_i=y_{i+1}$ for all $i\in\{1,\ldots, k-1\}$ (or $y_i=x_{i+1}$ for all $i\in\{1,\ldots, k-1\}$). Shift graphs were introduced by Erd\H{o}s and Hajnal~\cite{EH68},\cite{EH} and are standard examples of graphs with large chromatic number and large odd girth. More precisely, while the odd girth of $\Sh{n}{k}$ is $2k+1$, they proved\footnote{In~\cite{EH} authors considered infinite graphs, however their proof can be adapted for finite case (see~\cite{AKRR} and ~\cite{EtH} for more detailed description).} that $\Sh{n}{k}$ has chromatic number $(1+o(1))\log^{(k-1)}n$, where $\log^{(k-1)}$ stands for $k-1$ times iterated $\log_2$. 
	
	Shift graphs have another interesting property: For each finite set $G\subseteq V(\Sh{n}{k})$ the induced subgraph $\Sh{n}{k}[G]$ has a relatively large independent set with respect to $|G|$. In other words, the property ``having a large independent subset'' is hereditary for $\Sh{n}{k}$. Namely, for 
	\begin{equation}\label{eq:alphank}
	\alpha_{n}^{k}=\min\left\{\frac{\alpha(\Sh{n}{k}[G])}{|G|} \;:\; \emptyset\neq G\subseteq V(\Sh{n}{k})\right\},
	\end{equation}
	Erd\H{o}s, Hajnal and Szemer\'{e}di~\cite[Theorem 1]{EHS} proved the following.
	\begin{theorem}[Erd\H{o}s, Hajnal, Szemer\'{e}di]\label{thm:EHS}
		For positive integers $k<n$ 
		$$\alpha_{n}^k\geq \frac{1}{2}-\frac{1}{k}.$$
	\end{theorem}
	As for the upper bound, for $n\geq 2k+1 $ the shift graph $\Sh{n}{k}$ contains an odd cycle and so $\alpha_{n}^{k}<1/2$. Therefore, Theorem~\ref{thm:EHS} yields a lower bound which for large values of $k$ is essentially optimal.
	
	Nevertheless, determining the values of $\alpha_{n}^{k}$ for fixed $k$ and large $n$ seems to represent an interesting and non-trivial problem. We will concentrate our attention on the case $k=2$. In this case the bound from Theorem~\ref{thm:EHS} is not optimal, as we observe that $\alpha_{n}^2\geq 1/4$ for all $n$, and prove a matching upper bound.
	\begin{theorem}\label{thm:main}
		$\displaystyle \lim_{n\to\infty} \alpha_{n}^{2}=\frac{1}{4}.$
	\end{theorem}
	In~\cite{CEH}, Czipszer, Erd\H{o}s and Hajnal proved that the densest independent set of the infinite graph $\Sh{\N}{2}$ has density $1/4$ (see Section~\ref{sec:infinite} for precise formulation). We complement their result by showing that the infinite shift graph $\Sh{\N}{2}$ does not have a similar hereditary property, i.e., there exists $G\subseteq V(\Sh{\N}{2})$ such that any independent set in $\Sh{\N}{2}[G]$ has density zero in $G$ (see Theorem~\ref{thm:tree}).	
	%%%%%%%%%%%%%%%%%%%%%%%%%%%%%%%%%%%%%%%%%%%%%%%%%%%%%%%%%%%%
	%%%%%%%%%%%%%%%%%%%%%%%  Proof of Theorem   %%%%%%%%%%%%%%%%
	%%%%%%%%%%%%%%%%%%%%%%%%%%%%%%%%%%%%%%%%%%%%%%%%%%%%%%%%%%%%
	
	\section{Proof of Theorem~\ref{thm:main}} 
	Note that $\alpha^{2}_n=\min\left\{\frac{\alpha(\Sh{n}{2}[G])}{|G|}\;:\; \emptyset\neq G\subseteq V(\Sh{n}{2})\right\}$ is a nonincreasing positive sequence, so the sequence $\{\alpha^{2}_n\}$ has a limit. Additionally, we will often view $G\subseteq V(\Sh{n}{2})$ as a graph with $V(G)=[n]$ and set of edges equal to $G$. Subsequently $|G|$ will denote both a size of $G$ as a subset of $V(\Sh{n}{2})$, and the number of edges in $G$ when it is viewed as a graph.
	\subsection{Lower bound} 
	We first show that the value of the limit in Theorem~\ref{thm:main} is at least $1/4$.
	\begin{claim}\label{claim:lb}
		For every set $G\subseteq V(\Sh{n}{2})$ we have $\alpha(\Sh{n}{2}[G])\geq \frac{1}{4}|G|$. 
	\end{claim} 
	\begin{proof}
		Let $G\subseteq V(\Sh{n}{2})$ be given. Consider a random colouring $c:[n]\to \{r,b\}$ such that every $i\in[n]$ is coloured red/blue with probability $1/2$ independently of other elements of $[n]$. 
		
		Let $\Gc$ be a random subset of $G$ defined by 
		\begin{equation*}
		\Gc=\{(i,j)\in G \; : \; i<j,\; c(i)=b,\; c(j)=r\}.
		\end{equation*}
		Then such $\Gc$ is always an independent set in $\Sh{n}{2}$. Moreover, $\mathbb{P}(e\in \Gc)=\frac{1}{4}$ for every $e\in G$, and so $\mathbb{E}(|\Gc|)=\frac{1}{4}|G|$. Therefore $\alpha(\Sh{n}{2}[G])\geq \frac{1}{4}|G|$.
	\end{proof}
	
	\subsection{Upper bound} We now proceed and prove the upper bound
	%\begin{equation}\label{ineq:ub}
	% \liminf_{n\to\infty}\left\{\frac{\alpha(\Sh{n}{2}[G])}{|G|}\;:\;\emptyset\neq G \subseteq \Sh{n}{2}\right\}\leq \frac{1}{4}.
	%\end{equation}
	\begin{equation}\label{ineq:ub}
	\lim_{n\to\infty}\alpha^{2}_n\leq \frac{1}{4}.
	\end{equation}

	In what follows for every $\eps>0$, integer $d$ satisfying $\frac{3+\ln d}{4d}\leq \frac{\eps}{2}$, and for every integer $n\geq n_0(\eps,d)$ that is a multiple of $2^d$, we will construct a graph $\M{n,d}\subseteq V(\Sh{n}{2})$ with 
	\begin{equation*}
	\alpha(\Sh{n}{2}[\M{n,d}])\leq \left( \frac{1}{4}+\eps\right)|\M{n,d}|.
	\end{equation*}
	To be more precise, for such $\eps$ and $d$ we inductively build $\M{n,d}$ satisfying
	\begin{equation}\label{eq:ub}
	\frac{\alpha(\Sh{n}{2}[\M{n,d}])}{|\M{n,d}|}\leq \left( \frac{1}{4}+\frac{3+\ln d}{4d}+\frac{\eps}{2}\right).
	\end{equation}
	Since $\{\alpha_{n}^{2}\}$ is nonincreasing, (\ref{eq:ub}) implies that $\lim_{n\to \infty }\alpha_{n}^{2}\leq 1/4+\eps$, which subsequently implies (\ref{ineq:ub}) by letting $\eps\to 0$.
	
	While constructing $\M{n,d}$ we will use random bipartite graphs. Recall that if $G$ is a graph and $X, Y\subseteq V(G)$ then $G[X,Y]$ is a graph consisting of edges of $G$ with one vertex in $X$ and another in $Y$. Finally let $e_G(X,Y)=|E(G[X,Y])|$ and we will omit subscript when $G$ is obvious from the context. The following claim can be easily verified by considering a random graph and so the proof of Claim~\ref{claim:bip} is postponed to Appendix.
	
	\begin{claim}\label{claim:bip}
		For $\eps>0$ and $d \in \N$ there is $n_0=n_0(\eps,d)$ such that for all $n\geq n_0$ that are divisible by $2^d$ the following holds. Let $[n]=\S\cup \L$, where $\S=\{1,\ldots, \frac{n}{2}\}$ and $\L=[n]\setminus S$. There exists a bipartite graph $\Bnd$ with bipartition $V(\Bnd)=\S\sqcup \L$ such that
		\begin{itemize}
			\item[(i)]  $|\Bnd|=\frac{n^2}{2^{d+1}}.$
			\item[(ii)] for all $X\subseteq \S$ and $Y\subseteq \L$
			$$e(X,Y)=\frac{1}{2^{d-1}}|X||Y|\pm \frac{\eps n^2}{2^{d+2}}.$$
		\end{itemize} 
	\end{claim}\vspace{5pt}
	
	{\bf Construction of $\M{n,d}$.}

	\begin{definition}\label{def:M}
		For every even $n$ let $\M{n,1}$ be such that 
		$$\M{n,1}=\{(i,j) \;:\; 1\leq i\leq \frac{n}{2}<j\leq n\},$$
		i.e., $\M{n,1}$ is a complete balanced bipartite graph.
		
		For $d\in \N$ define graph $\M{n,d}$ recursively for all sufficiently large\footnote{$n\geq 2^in_0(\eps,d-i)$ for all $i\in\{0,1,\ldots, d-2\}$, where $n_0(\eps,d-i)$ is the number provided by Claim~\ref{claim:bip}.} $n$ such that $2^d|n$. Let $[n]=\S\cup \L$, where $\S=\{1, \ldots, \frac{n}{2}\}$ and $\L=[n]\setminus \S$. Then define
		$$\M{n,d}=\M{\S,d-1}\cup\M{\L,d-1}\cup\Bnd,$$
		where $\M{\S,d-1}=\M{\frac{n}{2},d-1}$, $V(\M{\L,d-1})=L$ and $\M{\L,d-1}\cong \M{\frac{n}{2},d-1}$, and $\Bnd$ is a graph guaranteed by Claim~\ref{claim:bip}.
	\end{definition}
	To summarize, every $\M{n,d}=G$ satisfies the following properties (for $\S_n=\{1,\ldots, \frac{n}{2}\}$ and $\L_n=[n]\setminus \S_n$):
	\begin{itemize}
		\item[(i)]  $\displaystyle e_{G}(\S_n, \L_n)=\frac{n^2}{2^{d+1}}.$
		\item[(ii)] for all $X\subseteq \S_n$ and $Y\subseteq \L_n$
		$$e(X,Y)=\frac{1}{2^{d-1}}|X||Y|\pm \frac{\eps n^2}{2^{d+2}}.$$
		\item[(iii)] $G[\S_n]\cong G[\L_n]=\M{\frac{n}{2},d-1}$
	\end{itemize}
	
	Using properties (i) and (iii) and induction on $d$ it is easy to verify that for all $d\in \N$ and $n$ divisible by $2^d$ 
	\begin{equation}\label{eq:|E(G)|}
	|\M{n,d}|=d\frac{n^2}{2^{d+1}}.
	\end{equation}		
	
	We will now proceed with proving~(\ref{eq:ub}). First let $G\subseteq V(\Sh{n}{2})$ and let $I\subseteq G$ be an independent set in $\Sh{n}{2}$. In other words there is no $1\leq i<j<k\leq n$ with both $(i,j)$ and $(j,k)$ in $I$. One can observe that for each such $I\subseteq G$ there exists a 2-colouring $c: [n] \to \{r, b\}$ with   $c(i)=r$ and $c(j)=b$ whenever $(i,j)\in I$, and then
	\begin{equation}\label{eq:\Gc}
	I\subseteq \Gc=\{(x,y)\in G \;:\; x<y,\; c(x)=b,\; c(y)=r\}.
	\end{equation}
	Therefore, in order to prove~(\ref{eq:ub}) we will show that for $G=\M{n,d}$ and any $c:[n]\to \{r, b\}$
	\begin{equation}\label{eq:Gc}
	\frac{|\Gc|}{|G|}\leq \frac{1}{4}+\frac{3+\ln d}{4 d}+\frac{\eps}{2}.
	\end{equation}
	
	For the rest of our calculation let $\eps$ be fixed. We will now prove~(\ref{eq:Gc})	by induction on $d$. In order to make use of recursive structure of $\M{n,d}$ we will prove a version of~(\ref{eq:Gc}) with an additional assumption that $|\{i \;:\; c(i)=b\}|=\alpha n$. 
	
	To that end for $d\in \N$, $\alpha\in [0,1]$ and $n\geq n_0(\eps, d)$ let 
	\begin{equation}\label{def:f_d}
	\fdl{d}{\alpha}(n)=d\cdot \max_{c} \left\{ \frac{|\Gc|}{|G|} \; : \; G=\M{n,d},\; |\{i \;:\; c(i)=b\}|=\alpha n \right\}.
	\end{equation}
	
	We will prove the following estimate on $\fdl{d}{\alpha}(n)$.
	\begin{claim}\label{claim:main}
		For every $d\in \N$, $\alpha\in [0,1]$ and $n\geq n_0(\eps, d)$
		$$\fdl{d}{\alpha}(n)\leq (d+3)(\alpha-\alpha^2)+\frac{1}{4}\ln d+\frac{d\eps}{2}.$$
	\end{claim}
	From (\ref{def:f_d}) it follows that for $G=\M{n,d}$ and any colouring $c$ we have 
	$$\frac{|\Gc|}{|G|}\leq \max_{\alpha\in[0,1]}\frac{\fdl{d}{\alpha}(n)}{d}.$$ Then by Claim~\ref{claim:main} we get 
	$$\frac{|\Gc|}{|G|}\leq \frac{1}{4}\frac{d+3}{d}+\frac{\ln d}{4d}+\frac{\eps}{2},$$
	establishing~(\ref{eq:Gc}) and (\ref{eq:ub}). Hence it remains to prove Claim~\ref{claim:main} in order to finish the proof of the upper bound.
	\begin{proof}[Proof of Claim~\ref{claim:main}]
		We prove a slightly stronger inequality for all $n\geq n_0(\eps, d)$
		\begin{equation}\label{ineq:fdl}
		\fdl{d}{\alpha}(n)\leq (d+3)(\alpha-\alpha^2)+\frac{1}{4}\sum_{i=3}^{d+1}\frac{1}{i}+\frac{d\eps}{2}.	
		\end{equation}
		The proof is by induction on $d$. For $d=1$ recall that $G=\M{n,1}$ is a complete bipartite graph between $\S_n$ and $\L_n$. Let $c:[n]\to \{r,b\}$ be such that for $B=\{i \;:\; c(i)=b\}$ we have $|B|=\alpha n$. Then in view of~(\ref{eq:\Gc}) the maximum value of $|\Gc|$ is achieved when $B=[\alpha n]$ and so
		$$\fdl{1}{\alpha}(n)=
		\begin{cases}
		2\alpha, \quad \alpha\in [0,\frac{1}{2}]\\
		2-2\alpha, \quad \alpha\in [\frac{1}{2},1].
		\end{cases}$$ 
		Now it is easy to verify that $\fdl{1}{\alpha}\leq 4(\alpha-\alpha^2)$ for all $\alpha \in [0,1]$, establishing~(\ref{ineq:fdl}) in the case $d=1$. 
		
		To prove inductive step let $G=\M{n,d}$ and let $c:[n]\to \{r,b\}$ be such that for $B=\{i \;:\; c(i)=b\}$ we have $|B|=\alpha n$. As before, let $\S=\{1, \ldots, \frac{n}{2}\}$ and $\L=[n]\setminus \S$. Let $B_\S$, $B_\L$, $R_\S$ and $R_{\L}$ denote the set of blue and red vertices in $\S$ and $\L$ respectively. We will further refine our analysis by assuming that $|B_\S|=x\frac{n}{2}$ with some $x\in [0,2\alpha]$. Since $|B_S|+|R_S|=\frac{n}{2}$, $|B_\S|+|B_\L|=|B|=\alpha n$, and  $|B_\L|+|R_\L|=\frac{n}{2}$, we have $|R_S|=(1-x)\frac{n}{2}$, $|B_\L|=(2\alpha-x)\frac{n}{2}$, and consequently $|R_\L|=(1-2\alpha+x)\frac{n}{2}$ (see Figure~\ref{fig:1}). Then
		\begin{equation}\label{eq:EGC}
		|\Gc|=|\Gc[\S]|+|\Gc[\L]|+e_{G}(B_\S, R_\L).
		\end{equation}
		\begin{figure}[H]
			\begin{center}
				\begin{tikzpicture}[line cap=round,line join=round,>=triangle 45,x=1.0cm,y=1.0cm, scale=1]
				\clip(-5.5,-2.6) rectangle (5.5,2);
				\begin{scope}[xshift=-1.5cm]
				\def\X{1};
				\def\Y{2};
				\def\A{-30}
				\filldraw[color=blue, opacity=0.5]
				({cos(180-\A)*\X}, {sin(180-\A)*\Y}) arc (180-\A:360+\A:{\X} and {\Y})--cycle; 
				\filldraw[color=red, opacity=0.5]
				({cos(\A)*\X}, {sin(\A)*\Y}) arc (\A:180-\A:{\X} and {\Y}) -- cycle;
				\draw[color=black] ({0},{(-1+sin(\A))*0.5*\Y}) circle (0pt) node {$x$};
				\draw[color=black] ({0},{(1+sin(\A))*0.5*\Y}) circle (0pt) node {$1-x$};
				\draw[color=black] ({0},{-1.2*\Y}) circle (0pt) node {$\Gc[S]$};
				\draw[color=black] ({-1.5-\X},{0}) circle (0pt) node {$G=\M{n,d}:$};
				\end{scope}
				
				\begin{scope}[xshift=1.5cm]
				\def\X{1};
				\def\Y{2};
				\def\A{-10}
				\filldraw[color=blue, opacity=0.5]
				({cos(180-\A)*\X}, {sin(180-\A)*\Y}) arc (180-\A:360+\A:{\X} and {\Y})--cycle; 
				\filldraw[color=red, opacity=0.5]
				({cos(\A)*\X}, {sin(\A)*\Y}) arc (\A:180-\A:{\X} and {\Y}) -- cycle;
				\draw[color=black] ({0},{(-1+sin(\A))*0.5*\Y}) circle (0pt) node {$2\alpha-x$};
				\draw[color=black] ({0},{(1+sin(\A))*0.2*\Y}) circle (0pt) node {$1-2\alpha+x$};
				\draw[color=black] ({0},{-1.2*\Y}) circle (0pt) node {$\Gc[L]$};
				\end{scope}
				
				\end{tikzpicture}
			\end{center}
			\caption{Proportions of red and blue vertices in $\Gc[S]$ and $\Gc[L]$.}\label{fig:1}
		\end{figure}
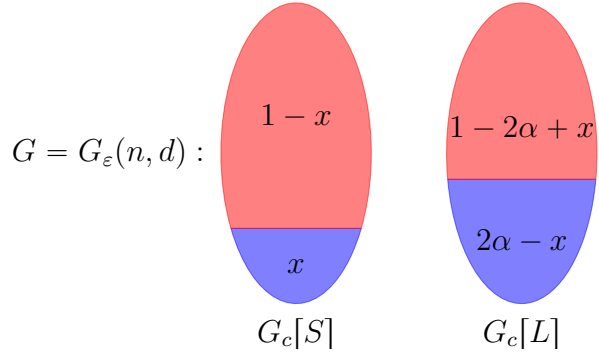
		
		Now, by (iii) $G[\S]=\M{\frac{n}{2},d-1}$ and we assumed $|B_{\S}|=x\frac{n}{2}$, so 
		\begin{equation}\label{eq:EGC2}
		|\Gc[\S]|\stackrel{(\ref{def:f_d})}{\leq} \frac{\fdl{d-1}{x}(\frac{n}{2})}{d-1}|G[\S]|\stackrel{(\ref{eq:|E(G)|})}{=}\frac{n^2}{2^{d+2}}\fdl{d-1}{x}\left(\frac{n}{2}\right).
		\end{equation}
		Similarly, since $|B_{\L}|=(2\alpha-x)\frac{n}{2}$ we have
		\begin{equation}\label{eq:EGC3}
		|\Gc[\L]|\leq \frac{n^2}{2^{d+2}}\fdl{d-1}{2\alpha -x}\left(\frac{n}{2}\right).
		\end{equation}
		And finally, since $G=\M{n,d}$,
		\begin{equation}\label{eq:EGC1}
		e_{G}(B_\S,R_\L) \stackrel{(ii)}{\leq} \frac{1}{2^{d-1}}|B_\S||R_\L|+\frac{\eps n^2}{2^{d+2}}=\frac{n^2}{2^{d+1}}\left(x(1-2\alpha+x)+\frac{\eps}{2}\right).
		\end{equation} 
		Combining (\ref{eq:EGC}) with (\ref{eq:EGC2}), (\ref{eq:EGC3}), and (\ref{eq:EGC1}) we obtain
		$$|\Gc|\leq \frac{n^2}{2^{d+1}}\left(\frac{1}{2}\left(\fdl{d-1}{x}\left(\frac{n}{2}\right)+\fdl{d-1}{2\alpha-x}\left(\frac{n}{2}\right)\right)+x(1-2\alpha+x)+\frac{\eps}{2}\right).$$
		Finally, $|G|=|\M{n,d}|\stackrel{(\ref{eq:|E(G)|})}{=}d\frac{n^2}{2^{d+1}}$ and so by (\ref{def:f_d}) we deduce that 
		\begin{equation}\label{eq:recursive}
		\fdl{d}{\alpha}(n)\leq \max_{x\in \mathbb{R}}\left\{\frac{1}{2}\left(\fdl{d-1}{x}\left(\frac{n}{2}\right)+\fdl{d-1}{2\alpha-x}\left(\frac{n}{2}\right)\right)+x(1-2\alpha+x)+\frac{\eps}{2}\right\}.
		\end{equation}
		
		The last inequality allows us to incorporate induction hypothesis. In particular, by induction hypothesis we have 
		\begin{align*}
		\fdl{d-1}{x}(\frac{n}{2})&\leq (d+2)(x-x^2)+\frac{1}{4}\sum_{i=3}^{d}\frac{1}{i}+\frac{(d-1)\eps}{2},\\
		\fdl{d-1}{2\alpha -x}(\frac{n}{2})&\leq (d+2)(2\alpha-x)(1-2\alpha+x)+\frac{1}{4}\sum_{i=3}^{d}\frac{1}{i}+\frac{(d-1)\eps}{2},
		\end{align*}
		and these two inequalities together with (\ref{eq:recursive}), after some simple but tedious algebraic manipulations yield
		\begin{equation*}
		\fdl{d}{\alpha}(n)\leq \max_{x\in \mathbb{R}}\left\{-(d+1)x^2+(1+2\alpha(d+1))x+(d+2)(\alpha-2\alpha^2)+\frac{1}{4}\sum_{i=3}^{d}\frac{1}{i}+\frac{d\eps}{2}\right\}.
		\end{equation*}
		In other words $\fdl{d}{\alpha}(n)\leq \max_{x\in \mathbb{R}}\{g(x)\},$ where $g(x)=ax^2+bx+c$ with $a=-(d+1)$. Since $a<0$ we have $\max_{x\in \mathbb{R}} g(x)=g(\frac{-b}{2a})=c-\frac{b^2}{4a}$. Therefore after another set of algebraic manipulations we obtain
		$$\fdl{d}{\alpha}(n)\leq \max_{x\in \mathbb{R}}\{g(x)\}\leq  (d+3)(\alpha-\alpha^2)+\frac{1}{4}\sum_{i=3}^{d+1}\frac{1}{i}+\frac{d\eps}{2},$$
		finishing the proof of the inductive step and Claim~\ref{claim:main}. 
	\end{proof}

	\section{Infinite graphs}\label{sec:infinite}
	Recall that Theorem~\ref{thm:main} states 
	\begin{equation}\label{eq:subfin}
	\displaystyle \lim_{n\to\infty} \min\left\{\frac{\alpha(\Sh{n}{2}[G])}{|G|}\;:\;\emptyset\neq G \subseteq V(\Sh{n}{2})\right\}=\frac{1}{4}.
	\end{equation}
	On the other hand, considering $I=\{(i,j) \;:\; 1\leq i\leq \frac{n}{2}<j \leq n\}$ we clearly have $\alpha(\Sh{n}{2})\geq \lfloor \frac{n^2}{4} \rfloor$. Moreover  $\lfloor\frac{n^2}{4}\rfloor$  is optimal, since any graph $G\subseteq V(\Sh{n}{2})$ with $|G|\geq \lfloor\frac{n^2}{4}\rfloor+1$ contains a triangle and hence such $G$ is not an independent set in $\Sh{n}{2}$. Therefore,
	\begin{equation}\label{eq:Knfin}
	\lim_{n \to \infty}\frac{\alpha(\Sh{n}{2})}{|\Sh{n}{2}|}=\frac{1}{2}.
	\end{equation}
	%Indeed, $H$ defined by $H=\{(i,j) | 1\leq i\leq \frac{n}{2}< j\leq n\}$ is independent in $\Sh{n}{2}$ and $|H|=\lfloor \frac{n^2}{4} \rfloor$. Moreover  $1/2$ is optimal, as any $G\subseteq V(\Sh{n}{2})$ with $|G|\geq \lfloor\frac{n^2}{4}\rfloor+1$, when viewed as graph, contains a triangle and hence is not independent in $\Sh{n}{2}$.
	
	It may be interesting to note that infinite version of~(\ref{eq:Knfin}) was considered by Czipszer, Erd\H{o}s and Hajnal~\cite{CEH} who proved that if $I$ is independent set in countable shift graph $\Sh{\N}{2}$, then the density of $I$ does not exceed $1/4$, i.e.
	\begin{equation}\label{eq:Kninf}
	\liminf_{n \to \infty}\frac{\left|I\cap \binom{[n]}{2}\right|}{\binom{n}{2}}\leq\frac{1}{4}.
	\end{equation}
	(Here $\frac{1}{4}$ is clearly optimal, since $I=\{(i,j) \;:\; i<j,\; i \text{ odd, } j \text{ even}\}$ is independent in $\Sh{\N}{2}$.)
	
	To complete this discussion we provide an infinite variant of~(\ref{eq:subfin}). 
	\begin{theorem}\label{thm:tree}
		There is $G\subseteq V(\Sh{\N}{2})$ such that if $I$ is an independent set in $\Sh{\N}{2}[G]$, then 
		$$\liminf_{n\to\infty}\frac{\left|I\cap \binom{[n]}{2}\right|}{\left|G\cap \binom{[n]}{2}\right|}=0.$$
	\end{theorem}
	\begin{proof}
		Consider an infinite ordered tree $G$ with $V(G)=\mathbb{N}$, and with  vertices labeled $v_{i}^j$, where $j$ denotes the ``level'' $L_j$ that vertex $v_{i}^j$ belongs to and $i$ denotes the order in which vertices are listed on the level.
		
		Consider a labeling of vertices of $G$ by integers satisfying $v_{i}^j<v_{i}^{j^\prime}$ if $j<j^\prime$ and $v_{i}^j<v_{i^\prime}^{j}$ if $i<i^\prime$ such that for all $v_{i}^j$ the finite set $N^{+}(v_{i}^j)$ of all children of $v_i^j$ forms an interval (and these intervals on the level $L_{j+1}$ follow the order of their parents on $L_j$, see Figure~\ref{fig:2}). Finally we will assume that for all $v_{i}^j$
		\begin{equation}\label{eq:Nfv}
		|N^{+}(v_i^{j})|\geq 2^{j}\sum_{v<v_{i}^j}|N^+(v)|.
		\end{equation}
		Now, let $I\subseteq G$ be an infinite independent set in $\Sh{\mathbb{N}}{2}$ and let $(v_{k}^{j-1}, v_{i}^j)\in I$, where $v_{k}^{j-1}$ and $v_{i}^j$ are parent and child respectively. Let $w=\max\{N^{+}(v_{i}^j)\}$ be the largest son of $v_{i}^j$ and let $W=\{1,\ldots, w\}$ (see Figure~\ref{fig:2}). Then 
		\begin{equation}\label{eq:G[W]}
		G[W]=\bigcup_{v\leq v_i^{j}}\left\{(v,u)\;:\; u\in N^{+}(v)\right\}.
		\end{equation}
		\begin{figure}[H]
			\begin{center}
				\begin{tikzpicture}[line cap=round,line join=round,>=triangle 45,x=1.35cm,y=1.5cm, scale=1]
				\clip(-5.7,-0.85) rectangle (6.3,3.3);
				% Lines level by level
				
				%level 0
				\draw[color=blue, line width=10, opacity=0.25] (-0.2,3)--(0.2,3);
				
				%level 1
				\draw[color=red] (0,3)--(-1,2);
				\draw[color=gray] (0,3)--(1,2);
				\draw[color=blue, line width=10, opacity=0.25] (-1.2,2)--(1.2,2);
				
				%level 2
				\foreach \a in {-2,-1.5,-1,-0.5}
				\draw[color=gray] (\a,1)--(-1,2);
				\foreach \a in {0,2}
				\draw[color=red] (\a,1)--(1,2);
				\foreach \a in {0.2,0.4,...,1.8}
				\draw[color=red, line width=0.25pt] (\a,1)--(1,2) ;
				
				%level 3
				\def\x{-5}
				
				\draw[color=red] (\x,0)--(-2,1);
				\foreach \a in {0.1,0.2,...,0.9}
				\draw[color=red, line width=0.25] (\x+\a,0)--(-2,1);
				\draw[color=red] (\x+1,0)--(-2,1);
				
				\def\x{-3.5}
				
				\draw[color=gray] (\x,0)--(-1.5,1);
				\foreach \a in {0.1,0.2}
				\draw[color=gray, line width=0.25] (\x+\a,0)--(-1.5,1);
				\foreach \a in {0.3,0.4,...,0.8}
				\draw[color=red, line width=0.25] (\x+\a,0)--(-1.5,1);
				\foreach \a in {0.9,1.0,...,1.1}
				\draw[color=gray, line width=0.25] (\x+\a,0)--(-1.5,1);
				\draw[color=gray] (\x+1.2,0)--(-1.5,1);
				
				\draw[color=gray] (\x+3,0)--(0,1);
				\foreach \a in {0.1,0.2,...,1.9}
				\draw[color=gray, line width=0.25] (3+\x+\a,0)--(0,1);
				\draw[color=gray] (\x+5,0)--(0,1);
				\draw[color=black] (\x+4,0) ellipse (1.2 and 0.15);
				
				\draw[color=gray] (\x+6.8,0)--(2,1);
				\foreach \a in {0.1,0.2,...,2.5}
				\draw[color=gray, line width=0.25] (6.8+\x+\a,0)--(2,1);
				\draw[color=gray] (\x+9.4,0)--(2,1);
				
				% Points level by level
				% level 0
				\draw[color=black, fill=black] (0,3) circle (2pt);
				\draw[color=blue, opacity=0.5] (-0.5,3) circle (0pt) node {$L_0$};

				% level 1
				\draw[color=black, fill=black] (-1,2) circle (2pt);
				\draw[color=black, fill=black] (1,2) circle (2pt);
				\draw[color=black] (1,2)++(0.5,0.2) circle (0pt) node {$v_{k}^{j-1}$};
				\draw[color=blue, opacity=0.5] (-1.5,2) circle (0pt) node {$L_1$};
				% level 2
				\foreach \a in {-2,-1.5,...,-0.5}
				\draw[color=black, fill=black] (\a,1) circle (2pt);
				
				\draw[color=black, fill=black] (0,1) circle (2pt);
				\draw[color=black] (0,1)++(-0.1,0.3) circle (0pt) node {$v_{i}^{j}$};
				\foreach \a in {0.2,0.4,...,1.8}
				\draw[color=black, fill=black] (\a,1) circle (0.5pt);
				\draw[color=black, fill=black] (2,1) circle (2pt);
				
				\draw[color=blue, line width=10, opacity=0.25] (-2.2,1)--(2.2,1);
				\draw[color=blue, opacity=0.5] (-2.5,1) circle (0pt) node {$L_2$};

				% level 3
				\def\x{-5}
				\draw[color=black, fill=black] (\x,0) circle (2pt);
				\foreach \a in {0.1,0.2,...,0.9}
				\draw[color=black, fill=black] (\x+\a,0) circle (0.5pt);
				\draw[color=black, fill=black] (\x+1,0) circle (2pt);
				
				\def\x{-3.5}
				\draw[color=black, fill=black] (\x,0) circle (2pt);
				\foreach \a in {0.1,0.2,...,1.1}
				\draw[color=black, fill=black] (\x+\a,0) circle (0.5pt);
				\draw[color=black, fill=black] (\x+1.2,0) circle (2pt);
				
				\foreach \a in {0.65,0.85,1.05}
				\draw[color=black, fill=black] (1.2+\x+\a,0) circle (0.5pt);
				\foreach \a in {0.65,0.85,1.05}
				\draw[color=black, fill=black] ({0.5*(1.2+\x+\a)-0.4},0.5) circle (0.5pt);			
				
				\draw[color=black, fill=black] (\x+3,0) circle (2pt);
				\foreach \a in {0.1,0.2,...,1.9}
				\draw[color=black, fill=black] (3+\x+\a,0) circle (0.5pt);
				\draw[color=black, fill=black] (\x+5,0) circle (2pt);
				
				\draw[color=black] (\x+5,0)++(0,0.2) circle (0pt) node {$w$};
				\draw[color=black] (\x+4.5,0)++(0,-0.4) circle (0pt) node {$N^{+}(v_{i}^j)$};

				\foreach \a in {0.65,0.85,1.05}
				\draw[color=black, fill=black] (5+\x+\a,0) circle (0.5pt);
				\foreach \a in {0.65,0.85,1.05}
				\draw[color=black, fill=black] ({1*(5+\x+\a)-0.75},0.5) circle (0.5pt);			
				
				\draw[color=black, fill=black] (\x+6.8,0) circle (2pt);
				\foreach \a in {0.1,0.2,...,2.5}
				\draw[color=black, fill=black] (6.8+\x+\a,0) circle (0.5pt);
				\draw[color=black, fill=black] (\x+9.4,0) circle (2pt);					
				
				\draw[color=blue, line width=10, opacity=0.25] (-5.2,0)--(6.1,0);
				\draw[color=blue, opacity=0.5] (-5.5,0) circle (0pt) node {$L_3$};
				
				% other levels
				\foreach \a in {0.2,0.4,...,0.6}
				\draw[color=black, fill=black] (-3,-0.2-\a) circle (0.5pt);
				\foreach \a in {0.2,0.4,...,0.6}
				\draw[color=black, fill=black] (0,-0.2-\a) circle (0.5pt);
				\foreach \a in {0.2,0.4,...,0.6}
				\draw[color=black, fill=black] (3,-0.2-\a) circle (0.5pt);
				
				%Set W
				\def\a{0.2}
				\draw[color=green, line width=1, opacity=0.5] (\a,3+\a)--({-3*\a},3+\a)--({-5-1.6*\a},\a)--({-5-1.6*\a},0-\a)--(1.5,0-\a)--({2+2*\a},1)--cycle;
				\draw[color=green] (-3,1.5+0.5) circle (0pt) node {$W$};
				\end{tikzpicture}
			\end{center}
			\caption{Infinite tree $G$, vertices are ordered top to bottom, left to right. Edges of $I$ are labeled with red.}\label{fig:2}
		\end{figure}
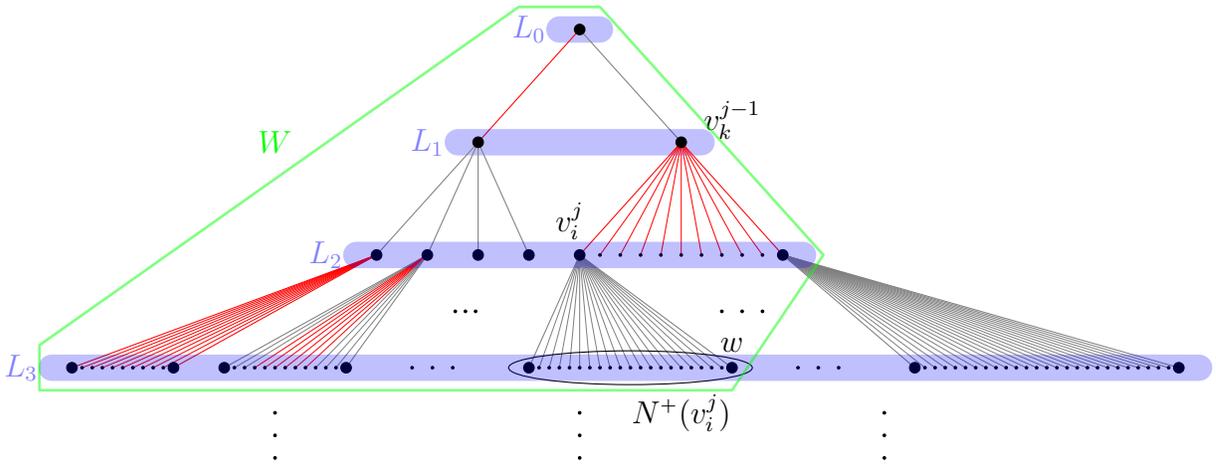
		In particular in view of~(\ref{eq:Nfv})
		\begin{equation}\label{eq:G[W].lb}
		|G[W]|\geq |\left\{(v_{i}^{j},u)\;:\; u\in N^{+}(v)\right\}|=|N^{+}(v_{i}^j)|.
		\end{equation}
		On other hand, since $(v_{k}^{j-1}, v_{i}^j)\in I$ and  $I$ is independent set in $\Sh{\mathbb{N}}{2}$, the set $I$ does not contain any other edge incident to $v_i^j$. Consequently,
		\begin{equation}\label{eq:I[W].ub}
		|I[W]|\leq \left|\bigcup_{v< v_i^{j}}\left\{(v,u)\;:\; u\in N^{+}(v)\right\}\right|\stackrel{(\ref{eq:Nfv})}{\leq}2^{-j}|N^{+}(v_{i}^j)|.
		\end{equation} 
		In view of~(\ref{eq:G[W].lb}) and~(\ref{eq:I[W].ub}) we have $|I[W]|/|G[W]|\leq \frac{1}{2^j}$. Now, since $I$ is infinite there are edges $(v_{k}^{j-1}, v_{i}^j)\in I$ with sufficiently large $j$, hence the ratio $|I[W]|/|G[W]|$ can be made arbitrary small, finishing the proof.
	\end{proof}
	\section{Concluding remarks}	
	In \cite{EHS} it was proved\footnote{the result follows from the proof of Theorem 1 in~\cite{EHS}} that for any $n, k$
	\begin{equation}\label{eq:EHS}
	\alpha_{n}^k\geq \begin{cases}
	\frac{1}{2}-\frac{1}{k}, \text{ if } k \text{ is even},\\
	\frac{1}{2}-\frac{1}{2k}, \text{ if } k \text{ is odd}.
	\end{cases}
	\end{equation}
	It remains an open problem to determine for any $k\geq 3$ the exact value of $\lim_{n\to \infty}	\alpha_{n}^{k}.$ For $k=4$ we were able to improve the constant in the lower bound~(\ref{eq:EHS}) from $\frac{1}{4}$ to $\frac{3}{8}$\footnote{$\alpha(\Sh{n}{4}[G])\geq \frac{3}{8}|G|$ can be proved by considering a random colouring $c:[n]\to\{0,1\}$ and forming an independent set in $\Sh{n}{4}$ by taking hyperedges of $G$ of form $1000$, $1110$, or $x01y$ for some $x,y\in\{0,1\}$.}  and for $k=3$ we believe that estimate in~(\ref{eq:EHS}) is sharp.
	% for k=4 consider a random 2-colouring of [n] and take quadruples of form x01x, 1000, 1110 from G to form G'. Then G' is independent set in a shift graph and contains on average 3/8|G| elements. 
	\begin{problem}
		Show that 
		$\lim_{n\to \infty}\alpha_{n}^{3}=\frac{1}{3}.$
	\end{problem}
	
	Finally, all of the results in this paper can be reformulated in terms of subgraphs with no increasing paths of length two. For instance, Theorem~\ref{thm:main} implies that for any $\eps>0$ there exists an vertex-ordered graph $G$ such that if $G'\subseteq G$ with $|G'|\geq \left(\frac{1}{4}+\eps\right)|G|$, then $G'$ contains an increasing path of length two, i.e. there are $i<j<k$ with $(i,j), (j,k)\in G'$. One can ask similar questions for longer increasing paths.
	
	\begin{problem}\label{prob:I3}
		For any $\eps>0$ does there exist an ordered graph $G$ such that if $G'\subseteq G$ with $|G'|\geq \left(\frac{1}{3}+\eps\right)|G|$, then $G'$ contains an increasing path of length three?
	\end{problem}
	Note that in regards to Problem~\ref{prob:I3}, one can consider a random coloring $c$ of $V(G)$ with colors $\{0,1,2\}$ and define $G'$ to be the collection of all $(i,j)\in E(G)$ with $i<j$ and $c(i)<c(j)$. Then such $G'$ on average contains $\frac{1}{3}|G|$ edges and has no increasing paths of length three, motivating constant $\frac{1}{3}$ in the problem.
	
	%%%%%%%%%%%%%%%%%%%%%%%%%%%%%%%%%%%%%%%%%%%%%%%%%
	%\subsection*{Acknowledgements}
	%We would like to thank anonymous referees for valuable comments.

	\section*{Appendix}\label{sec:appendix}
	\begin{proof}[Proof of Claim~\ref{claim:bip}]
		Let $\Bnd=G$, where $G$ is a random graph between $\S$ and $\L$ obtained by selecting a random subset of size $\frac{n^2}{2^{d+1}}$ without replacement from $K_{\S,\L}$ (complete bipartite graph between $\S$ and $\L$). Then $G$ satisfies (i) and we will show that $G$ satisfies (ii) almost surely.
		
		For every $X\subseteq \S$ and $Y\subseteq \L$, $e(X,Y)=e_G(X,Y)$ is distributed as a hypergeometric random variable $\text{H}\left(\frac{n^2}{4}, \frac{n^2}{2^{d+1}},|X||Y|\right)$ with expectation $\frac{1}{2^{d-1}}|X||Y|$. Let $B_{X,Y}$ be the event that $$\left|e_{G}(X,Y)-\frac{1}{2^{d-1}}|X||Y|\right|>\frac{\eps n^2}{2^{d+2}},$$ i.e., $B_{X,Y}$ is the event that (ii) fails for given $X$ and $Y$. 
		
		We will use a concentration inequality for hypergeometric random variables (this version is a corollary of Theorem 2.10 and inequalities (2.5),(2.6) of Janson, \L uczak, Rucinski~\cite{JLR}).
		\begin{theorem}\label{thm:Chernoff}
			Let $Z\sim \text{H}(N,m,k)$ be a hypergeometric random variable with the expectation $\mu=\frac{mk}{N}$, then for $t\geq 0$
			$$\mathbb{P}(|Z-\mu|> t)\leq 2 \exp\left(\frac{-t^{2}}{2(\mu+t/3)}\right).$$
		\end{theorem}
		
		For a given $X\subseteq L$ and $Y\subseteq R$, as a consequence of Theorem~\ref{thm:Chernoff} with $Z=e_{G}(X,Y)$, $t=\frac{\eps n^2}{2^{d+2}}$ and $\mu=\frac{1}{2^{d-1}}|X||Y|\leq \frac{n^2}{2^{d+1}}$ we get $$\mathbb{P}(B_{X,Y})=e^{-\Omega(n^2)},$$
		where constant in $\Omega()$ term depends on $\eps$ and $d$ only.
		Therefore,
		$$\mathbb{P}\left(\bigcup_{X,Y}B_{X,Y}\right)\leq \sum_{X,Y}\mathbb{P}\left(B_{X,Y}\right)\leq 2^ne^{-\Omega(n^2)}=o(1).$$
		In particular, $\mathbb{P}(G \text{ satisfies (ii)})=\mathbb{P}\left(\bigcap_{X,Y}\overline{B_{X,Y}}\right)=1-o(1)$. Hence, $G$ almost surely satisfies (ii).
	\end{proof}
	
\end{document}